\documentclass[10pt]{article}
\usepackage{graphicx}
\usepackage{latexsym}
\usepackage{stmaryrd}
\usepackage{amsmath}
\usepackage{amssymb}
\usepackage{amsthm}
\usepackage{multirow}
\usepackage{authblk}
\usepackage[spaces,hyphens]{url}
\usepackage{bm}
\usepackage{hyperref}
\usepackage{tikz}
\usepackage{multicol}
\usepackage{mathtools}

\def \N {{\mathbb N}}
\def \Z {{\mathbb Z}}
\def \R {{\mathbb R}}

\def \D {{\mathcal D}}
\def \E {{\mathcal E}}

\def \v {{\mathbf v}}

\def \Om{{\Omega}}
\def\om{{\omega}}
\def \Ga {{\Gamma}}

\newcommand*{\permcomb}[4][0mu]{{{}^{#3}\mkern#1#2_{#4}}}
\newcommand*{\comb}[1][-1mu]{\permcomb[#1]{C}}
\newtheorem{theorem}{Theorem}[section]
\newtheorem{cor}[theorem]{Corollary}
\newtheorem{lemma}[theorem]{Lemma}
\newtheorem{pro}[theorem]{Proposition}
\newtheorem{rem}[theorem]{Remark}

\newtheorem{definition}[theorem]{Definition}

\newtheorem{ex}[theorem]{Example}

\newcommand{\edge}{%
  \mathrel{-}
  \joinrel\joinrel 
  \mathrel{-}
}
\newcommand{\notedge}{%
  \mathrel{\mkern2mu}
  \arrownot 
  \mathrel{\mkern-2mu}
  \edge
}

\usepackage{setspace}
\usepackage[square,sort,comma,numbers]{natbib}
\usepackage[english]{babel}
\usepackage[euler]{textgreek}
\usepackage{stmaryrd}
\usepackage{mathtools}
\usepackage{pgf,tikz}
\usepackage{tikz}
\usepackage{tkz-graph}
\usepackage{array}
\tikzstyle{vertex}=[circle,draw, inner sep=0pt, minimum size=4pt] 
\newcommand{\vertex}{\node[vertex]}
\usetikzlibrary{decorations.markings}
\usepackage{lipsum}
\title{The Generating graph of Dicyclic Groups} 
\author[1]{Kavita Samant}
\author[2]{A. Satyanarayana Reddy}
\affil[1,2]{Department of Mathematics\\Shiv Nadar Institution of Eminence, Delhi-NCR, India.}
\affil[2]{Corresponding author. E-mail id(s): ks299@snu.edu.in;}
\affil[1]{Contributing author: satya.a@snu.edu.in}

\date{}

\begin{document}
\maketitle
\begin{abstract}
   For a group $G,$ the generating graph of $G,$ denoted by $\Ga(G).$ We define $Q_n=\langle x,y: x^{2n}=y^4=1, x^n=y^2,y^{-1}xy=x^{-1}\rangle,$ the dicyclic group of order $4n.$ This paper primarily delves into exploring the graph characteristics and spectral properties of various matrices associated with $\Ga(Q_n)$. Specifically, we determine the complete spectrum of the adjacency, Laplacian, distance, and eccentricity matrices. Additionally, we completely determine the spectrum pertaining to the distance and eccentricity matrices of the dihedral group of order $2n$, denoted as $D_n$.
\end{abstract}
\section{Introduction}
The spectral properties of graphs associated with groups are a potent tool for connecting the algebraic structure of groups with their combinatorial and geometric characteristics. These properties offer profound insights and enable various applications across theoretical and practical domains, including mathematics, computer science, and chemical science. For numerous graphs linked to group structures, such as Cayley graphs, commuting graphs, power graphs, and others, the spectral properties have been extensively studied, enhancing our comprehension of the underlying algebraic frameworks. However, the spectral properties of generating graphs have not been exhaustively investigated. 
Specifically, in this article, we delve into the spectral properties of the generating graph associated with the dicyclic and dihedral groups. In paper~\cite{Me}, the complete spectrum for the adjacency and Laplacian has been determined for the dihedral groups. Now our aim is to extend the idea for dicyclic groups.

An {\em undirected graph} \( X \) is defined as a pair \( (V(X), E(X)) \) where \( V(X) \) is a non-empty set called the vertex set, and \( E(X) \) is a subset of all unordered pairs of distinct elements from \( V(X) \), called the edge set.

A group $G$ is called {\em two-generated}, if a pair of elements can generate $G$. The {\em generating graph} of $G$, denoted by \(\Gamma(G)\), has $G$ as its vertex set, with edges between distinct vertices that generate $G$. Hence, generating graphs are considered only for two-generated groups; otherwise, they are empty.\\
The core idea, initially studied from a probabilistic perspective, is that two randomly selected elements from a finite group $G$ generate $G$ with probability \( P(G) \), defined as:
\[ P(G) = \frac{|\{(x, y) \in G \times G : \langle x, y \rangle = G \}|}{\comb{n}{2}}. \]
Many significant results for finite simple groups were explored using this probabilistic notion by Liebeck and Shalev~\cite{Lieback}, Guralnick and Kantor~\cite{Guralnick}, and Lucchini and Detomi~\cite{crown}, demonstrating that \(\Gamma(G)\) is often a rich graph. The structure of two-generated groups and their generating pairs, and the concept of generating graphs, were further defined by Lucchini and Maróti~\cite{lucchini2009clique,MR2816431}, who investigated various graph-theoretic properties of these graphs and posed several questions.

In this article, we will explore how the adjacency spectra of dihedral groups (see~\cite{Me}) can be utilized to determine the spectra of various associated matrices for both dihedral and dicyclic groups. Specifically, we will demonstrate that the eigenvalues of the adjacency and Laplacian matrices of the dicyclic group \( Q_n \) are exactly twice those of the corresponding dihedral group \( D_n \), while the multiplicities of the non-zero eigenvalues remain unchanged. Additionally, we will show that the Laplacian and eccentricity eigenvalues are integer values for both dihedral and dicyclic groups. However, this integrality property does not extend to the adjacency and Laplacian eigenvalues in general. The quotient graphs are one of the tools in the determination of the spectrum of the associated matrices.

The article is organized into five sections. In Section~\ref{sec:pre}, we review essential definitions and results. Section~\ref{sec:Rel} we discuss the graph properties for the dicyclic group and demonstrates the graph relations between dicyclic and dihedral groups. We will show that the spectrum of the dicyclic group's generating graph can be derived from that of the dihedral group.

In Section~\ref{sec4:spec}, we determine the spectra of adjacency and Laplacian matrices for the dicyclic group. Finally, in the last section~\ref{sec:EDspec}, we detemine the spectrum of the distance matrix and a distance-related matrix called the eccentricity matrix for the dihedral group, and subsequently derive these for the dicyclic group.

\section{Preliminaries} \label{sec:pre}
In this section, we provide a selection of definitions and notations that will be used throughout the paper. Our focus is solely on simple, undirected graphs.

Let $X$ be a graph, and let $V(X)$ and $E(X)$ be the vertex set and the edge set of $X$ respectively. The degree of any vertex $v\in X,\,\deg(v)$ is the number of the edges incident to $v.$ The distance $d(u,v)$ denotes the length of the shortest path between the vertices $u$ and $v$ in $X.$ The diameter of $X$, denoted by $diam(X)$ is the maximum distance among all the pairs of vertices in $X.$ Some of the standard notations like $\omega(X),$ $\chi(X),$ $\alpha(X),$ $\gamma(X),$ $\gamma_t(X)$ denote the clique number, chromatic number, independence number, domination number and total domination number respectively. Instead of giving an exhaustive list of definitions, we refer the reader to Godsil and Royle for the graph theory basics~\cite{godsil}. For group theoretic concepts, any standard book can be followed.\par

We considered only finite groups in the paper. In particular, we focus on the family of dicyclic groups, characterized by being two-generated and non-abelian. 
The {\em Dicyclic} groups 
$$Q_{n}=\{x,y\,:\,x^{2n}=1,x^n=y^2,yx=x^{-1}y\}$$ is a group of order $4n,$ where $n\geq 2$ are the positive integers. The dicyclic groups are also known as generalised quaternion group when $n=2^k$ where $k\geq 2.$ The group can also be expressed in the following way
$$Q_{n}=\{x^a,x^by\,:\, 0\leq a,b\leq (2n-1)\}$$ with the same restrictions on elements. The group has exactly one element of order 2, that is $x^n=y^2.$ 
We introduce some notations here, that we will use throughout. We denote
   \begin{align*}
    R&=\langle x \rangle,  R_1=\{x^{a}\,:\, \gcd(n,i)=1 \text{ where }  0\leq a\leq (2n-1) \},  \\ R_2&=R\setminus R_{1} \text{ and }
 \Omega=\{x^{b}y\,:\, 0\leq b\leq (2n-1)\}.
   \end{align*}
 Thus, $|R_1|=2\varphi(n),$ $|R_2|=2(n-\varphi(n))$ and $|\Omega|=2n.$  $$Q_{n}=R_1\cup R_2\cup\Omega.$$
Now we review some group theoretic properties of $Q_n$ in the following proposition, which can be useful later.
\begin{pro}\label{pro:basic}
    Let $n=p_1^{e_1}p_2^{e_2}\dots p_{k}^{e_k}$ and $p's$ are distinct primes. Let $Q_n=\langle x,y\rangle,$ then it has the following properties.
    \begin{enumerate}
        \item For any $0\leq b\leq 2n-1,$ $(x^by)^2=y^2=x^n.$
        \item If $\gcd(a,n)=1,$ then $\langle x^a,x^n\rangle=R.$ 
    \item Dicyclic groups are solvable.
    \item The center of the dicyclic groups is the set $\{1,x^n\}.$
     \item The Frattini subgroup of a group $G,$ denoted by $\Phi(G)$ is the intersection of all maximal subgroups. For $G=Q_n,$ the subgroup $\Phi(G)=\langle x^{p_1p_2\dots p_k} \rangle.$  
    \end{enumerate}
\end{pro}
Now we define the generating set of the dicyclic group $Q_{n},$ for any $n>1,$ by
$$\text{Gen}(Q_n)=\{(g_1,g_2)\in Q_n\times Q_n\,|\, \langle g_1,g_2\rangle=Q_n\}.$$ What can we say about the probability that two distinct randomly chosen elements in \( Q_n \) generate the entire group? In general, the probabilistic approach to examining \( k \) randomly chosen generating elements of a group has been widely explored in the literature for various groups. Esther Banaian, in his article~\cite{Banaian}, discussed the generating behavior of elements in dicyclic groups specifically for \( k=2 \). Here, we will present some of his key findings related to this topic.
\begin{theorem}
In $Q_{n},$ the set $\{x^a,x^by\}$ will generate the whole group if and only if $\gcd(a,n)=1.$
\end{theorem}
\begin{theorem}
    In $Q_{n},$ the set $\{x^cy,x^dy\}$ with $c>d$ will generate the whole group if and only if $\gcd(c-d,n)=1.$
\end{theorem}
For the sake of completeness, let us explicitly understand the above results. 
Since $R=\langle x\rangle,$ is a proper subgroup of $Q_{n},$ thus any of its two elements can't generate the whole group. However, if we choose $x^a\in R$ and $x^by\in \Omega,$ to show that they generate $Q_n$ or not, it is sufficient to check that whether $R$ is inside the subgroup $\langle x^a,x^by\rangle$ or not. Therefore,  if $\gcd(a,2n)=1$, then clearly $R$ is inside it. If $\gcd(a,2n)\neq 1$ but $\gcd(a,n)=1,$ then we know that $(x^by)^2=y^2=x^n,$ and thus $\langle x^a,x^n\rangle=R.$ Lastly, if $\gcd(a,n)\neq 1,$ then $\langle x^a,x^n\rangle$ is a proper subgroup of $R$ and hence $\langle x^a,x^by\rangle$ will not generate $Q_n.$ Therefore $R_1\times \Om\subseteq \text{Gen}(Q_n)$ and clearly $R_2$ will not take part in the generation of group. Now if we consider elements from the set $\Omega,$ say $x^cy$ and $x^dy,$ then it is clear that $x^n\in \langle x^cy,x^dy\rangle.$ Note that $x^cyx^dy=x^nx^{(c-d)}$ and therefore $x^{(c-d)}\in\langle x^cy,x^dy\rangle.$ Note that $\langle x^{(c-d)},x^n\rangle=R$ if and only if $\gcd(c-d,n)=1.$ So we conclude that $\langle x^cy,x^dy\rangle=Q_n$ if and only if  $\gcd(c-d,n)=1.$

Using these results, we derive a general formula for the probability that two randomly chosen elements generate \( Q_n \) by finding the set $\text{Gen(n)},$ a formula that is not explicitly provided in the original paper [cite]. However, the author does fully determine this probability for cases where \( n = p \) and \( n = p^k \) with \( k > 1 \) and \( p \) being a prime.
\begin{theorem}
    The probability of generating $Q_n$ for any $n>2$ is given by 
    $$P(Q_n)=\frac{3\varphi(n)}{4n-1}.$$
\end{theorem}
\begin{proof}
    To determine the probability, we first need to explicitly identify the set \( \text{Gen}(Q_n) \). To achieve this, we will construct several subsets that clarify which elements of \( Q_n \) can pair together to generate the group. This approach will provide a clearer understanding of the generating pairs within \( Q_n \).
 For each $x^ay\in \Omega,$ where $0\leq a\leq 2n-1$, we define a set 
 \begin{equation}\label{eq:Ba}
    B_{a}=\{x^{b}y:\,b=a+i\, \text{ where } i\in U(n)\cup \left(n+U(n)\right)\},
 \end{equation}
  where $U(n)$ is the group of units of ring $\Z_{n}$ and $n+U(n)=\{n+u:u\in U(n)\}.$ Clearly, for each $a,$ $|B_a|=2\varphi(n).$ Using this, corresponding to each $a$ we define an another set
$$A_{a}=\{x^ay\}\times B_{a}.$$  We observe that for each $a,$ the set $A_{a}\subseteq \Omega\times \Omega$ and  $\bigcup\limits_{a=0}^{2n-1}A_{a}\subseteq \text{Gen(n)}.$ Therefore,
\begin{equation}\label{eq:gen1}
    \text{\text{Gen(n)}}=(\R_{1}\times \Omega)\cup(\Omega\times R_1)\cup(\bigcup\limits_{a=0}^{2n-1} A_{a}).
\end{equation}
Notice that $|R_{1}\times \Omega|=|\Omega\times R_1|=4n\varphi(n)$ and $|A_{a}|=4n\varphi(n)$ for each $a$. Since all these pairs are ordered, thus we consider them distinct. We have
\begin{equation}\label{eq:gen}
|\text{Gen(n)}|=12n\varphi(n).
\end{equation}
Consequently $P(Q_n),$ the probability function is given by 
$$P(Q_n)=\frac{|\text{Gen(n)}|}{2 {4n \choose 2}}=\frac{3\varphi(n)}{4n-1}.$$
\end{proof}
The above formula suggests that the probability of generating $Q_n$ converges to $\frac{3}{4}$ as $n$ grows larger.

\section{Relationship between dihedral and dicyclic generating graphs}\label{sec:Rel}
In this section, we begin by examining the structure of the generating graph for dicyclic groups. Additionally, we will review the concept of graph quotients, as these are essential tools in our analysis. We then establish a connection between the quotient graph of dicyclic groups and the generating graph of dihedral groups. Understanding this relationship will enable us to analyze the spectra of the adjacency matrices associated with their generating graphs. Reader can refer to~\cite{Me} to understand the graph structure and spectral properties of the generating graph of the dihedral group. 
\subsection{Generating graph of $Q_n$}
Let us extend the idea of the generation of the Dicyclic group to the generating graphs. 
We denote the generating graph of $Q_n$ by $\Ga(Q_n).$ To see the adjacency among the vertices in the graph, we need to find the generating pairs of the group. So we could say that there is an edge between $a,b\in Q_n$ if and only if $(a,b)\in \text{Gen}(Q_n).$ In particular, the edge set of $\Ga(Q_n)$ is $E(Q_n)=\{\{x,y\}\,:\,(x,y)\in \text{Gen}(Q_n)\}.$ By definition, the generating graph of a group is undirected.  Thus the number of edges in $\Ga(Q_n)$ is $6n\varphi(n)$ (using~\ref{eq:gen}).

Let $G$ be a group and let $S(G)=\{g\in G\,:\,\exists \, z\in G,\langle g, z\rangle = G\}.$ In case of $Q_n,$ we can see that $S(Q_n)=R_1\cup \Omega.$ Hence $\Ga(Q_n)$ is a disconnected graph, and the number of connected components equals $2(n-\varphi(n))+1.$ In other words, the subgraph of $\Ga(Q_n)$ induced by the vertices $S(Q_n)$ is one connected component, and others are the isolated vertices that correspond to the members of $R_2.$ We adopt the notation $\Ga_T$ to denote the subgraph of $\Ga(Q_n),$ which is induced by a subset $T$ of $V(Q_n).$ Thus we denote $\Ga_{S(Q_n)},$ as the subgraph induced by the set $S(Q_n).$  For notational convenience, 
we write $\Ga_{S(Q_n)}$ as $\Delta(Q_n),$ for other subsets we will be considering the mentioned notation. The following result explicity give the degree of each vertex of $\Ga(Q_n).$
\begin{theorem}
    In the graph $\Ga(Q_n),$ for any $g\in V(Q_n),$ the vertex degree is as follows:
    $$\deg(g)\in\{2n,4\varphi(n),0\}.$$
\end{theorem}
\begin{proof}
    Let $N(g)$ denotes the neighbourhood of a vertex $g.$ Then, using Equations~\ref{eq:Ba} and~\ref{eq:gen1}, we find $N(g)$ for each $g\in Q_n$ as below.
\begin{equation}\label{eq:nbd}
 N(g)=\begin{dcases}
       \Omega & \mbox{if $g\in R_1,$}\\
       R_1\cup B_a & \mbox{if $g=x^ay\in\Omega,$}\\
       \emptyset & \mbox{if $g\in R_2.$}\\
      \end{dcases}
\end{equation}
Thus, for any $g\in V(Q_n),$ $$\deg(g)\in \{2n,4\varphi(n),0\}.$$
\end{proof}
 
\begin{cor}
The generating graphs of $Q_n,$ $\Delta(Q_n)$ after removing the isolated vertices are regular if and only if  $n=2^{k}$ for any $k\in \N.$ Moreover, the degree of the regular graphs is $2^{k+1}.$
\end{cor}
\begin{proof}
    The vertex degree of $\Delta(Q_n)$ belongs to the set $\{2n, 4\varphi(n)\}.$ Note that $4\varphi(n)=2n$ if and only if $n=2^{k},$ where $k\geq 1$, thus in that case, the vertex degrees of $\Ga(Q_{2^k})$ are $0$ and $2^{k+1}$. In other words, $\Delta(Q_{2^k})$ is $2^{k+1}$-regular graph. 
\end{proof}
For a group $G,$ we denote $I(G)$ as the set of isolated vertices of the graph $\Ga(G),$ and let $\Phi(G)$ be the Frattini subgroup of $G,$ then we have the following  implication for the dicyclic groups.
\begin{cor}
   For $G=Q_n,$ then $I(G)=\Phi(G)$ if and only if $n=p^k$ where $p$ is a prime and $k>1$ is an integer.
\end{cor}

\begin{proof}
It is clear that $R_2$ is the set of isolated vertices in $\Ga(Q_n)$. Since $\Phi(G)=\langle x^{p_1p_2\dots p_k}\rangle$ where $p_i's$ are distinct prime factors of $n$ (see Proposition~\ref{pro:basic}) then, $\Phi(G)\cong Z_{2n/n_0}.$ Since $\Phi(G)$ is the set of non generators of a group, thus $\Phi(G)\subseteq I(G).$ Therefore, $\Phi(G)=I(G)$ if and only if $|\Phi(G)|=|I(G)|.$ So we have the following omplications.
\begin{align*}
    \frac{2n}{n_0}&=2(n-\varphi(n))\\
    &=2\left(n-\frac{n}{n_0}\varphi(n_0)\right)\\
    &=\frac{2n}{n_0}\left(n_0-\varphi(n_0)\right).
\end{align*}
Thus, we conclude that $n_0-\varphi(n_0)=1$ if and only if $n_0$ is a prime, and hence $n=p^k$ for some prime $p$ and $k\in \N.$
\end{proof}
\begin{rem}
    The above corollary implies that we have an infinite family of groups whose generating graph have isolated vertices outside its Frattini subgroup.
\end{rem}

 Let us see some examples below.
\begin{ex}
The graphs $Q_3$ and $Q_4$ are shown in Figure~\ref{fig:ga4}.
\begin{figure}[!ht]
	\centering
	\begin{tabular}{cc}
	\begin{tikzpicture}[scale=0.35]
       
	\draw [line width=0.5pt] (5,7.1)-- (5,-1);
	\draw [line width=0.5pt] (5,7.1)-- (1.5,2.5);
	\draw [line width=0.5pt] (5,7.1)-- (9,3);
    \draw [line width=0.5pt] (5,7.1)--(-2,1);
	\draw [line width=0.5pt] (5,7.1)-- (7.5,-1);
	\draw [line width=0.5pt] (5,7.1)-- (12.5,2);

	\draw [line width=0.5pt] (3,7.1)-- (5,-1);
	\draw [line width=0.5pt] (3,7.1)-- (1.5,2.5);
	\draw [line width=0.5pt] (3,7.1)-- (9,3);
    \draw [line width=0.5pt] (3,7.1)--(-2,1);
	\draw [line width=0.5pt] (3,7.1)-- (7.5,-1);
	\draw [line width=0.5pt] (3,7.1)-- (12.5,2);

    \draw [line width=0.5pt] (8.4,7.6)-- (5,-1);
	\draw [line width=0.5pt] (8.4,7.6)-- (1.5,2.5);
	\draw [line width=0.5pt] (8.4,7.6)-- (9,3);
    \draw [line width=0.5pt] (8.4,7.6)--(-2,1);
	\draw [line width=0.5pt] (8.4,7.6)-- (7.5,-1);
	\draw [line width=0.5pt] (8.4,7.6)-- (12.5,2);

	\draw [line width=0.5pt] (1,9)-- (5,-1);
	\draw [line width=0.5pt]  (1,9)-- (1.5,2.5);
	\draw [line width=0.5pt]  (1,9)-- (9,3);
    \draw [line width=0.5pt]  (1,9)--(-2,1);
	\draw [line width=0.5pt]  (1,9)-- (7.5,-1);
	\draw [line width=0.5pt]  (1,9)-- (12.5,2);
 
	\draw [line width=0.5pt] (1.5,2.5)-- (5,-1);	
	\draw [line width=0.5pt] (9,3)-- (1.5,2.5);
	\draw [line width=0.5pt] (5,-1)-- (9,3);

    \draw [line width=0.5pt] (1.5,2.5)-- (12.5,2);
    \draw [line width=0.5pt] (-2,1)-- (9,3);
    \draw [line width=0.5pt] (-2,1)-- (12.5,2);
    \draw [line width=0.5pt] (1.5,2.5)-- (5,-1);
    \draw [line width=0.5pt] (5,-1)-- (-2,1);
    \draw [line width=0.5pt] (12.5,2)-- (7.5,-1);
    \draw [line width=0.5pt] (12.5,2)-- (5,-1);
    \draw [line width=0.5pt] (9,3)-- (7.5,-1);

    \draw [line width=0.5pt] (1.5,2.5)-- (7.5,-1);
    \draw [line width=0.5pt] (-2,1)-- (7.5,-1);

    \draw (0.5,2.5) node[anchor=north west] {$y$};
	\draw (9,3) node[anchor=north west] {$yx$};
	\draw (4,-1) node[anchor=north west] {$yx^{2}$};
	\draw (3,9) node[anchor=north west] {$x^{2}$};
	\draw (6,7.6) node[anchor=north west] {$x$};
	\draw (7,9) node[anchor=north west] {$1$};

    \draw (-2,1) node[anchor=north west] {$yx^3$};
	\draw (6.5,-1) node[anchor=north west] {$yx^{5}$};
    \draw (11.5,2) node[anchor=north west] {$yx^4$};
	\draw (1,10.5) node[anchor=north west] {$x^5$};
	
	\draw (8.8,8) node[anchor=north west] {$x^4$};
	\draw (9,9.5) node[anchor=north west] {$x^3$};
	\begin{scriptsize}
	\draw [fill=black] (1.5,2.5) circle (2.5pt); 
	\draw [fill=black] (5,-1) circle (2.5pt);
    \draw [fill=black] (9,3) circle (2.5pt);
	\draw [fill=black](3,7.1)circle (2.5pt);

	\draw [fill=black](5,7.1)circle (2.5pt);
	\draw [fill=black](7,9)circle (2.5pt);

    \draw [fill=black](-2,1) circle (2.5pt);
	\draw [fill=black](7.5,-1) circle (2.5pt);
    \draw [fill=black](12.5,2) circle (2.5pt);
	\draw [fill=black](1,9)circle (2.5pt) ;
	
	\draw [fill=black](8.4,7.6)circle (2.5pt);
	\draw [fill=black](9,9) circle (2.5pt);
    
	\end{scriptsize}
	\end{tikzpicture}
	
  &\hspace*{10mm}
	\begin{tikzpicture}[scale=0.4]

    \draw [line width=0.5pt] (14.5,11)-- (10,7);
    \draw [line width=0.5pt] (14.5,11)-- (10,9);
    \draw [line width=0.5pt] (14.5,11)-- (9,8);
    \draw [line width=0.5pt] (14.5,11)-- (11,7.5);
    \draw [line width=0.5pt] (14.5,11)-- (8,11);
    \draw [line width=0.5pt] (14.5,11)-- (14,5);
    \draw [line width=0.5pt] (14.5,11)-- (5.5,8);
    \draw [line width=0.5pt] (14.5,11)-- (16,8);

    \draw [line width=0.5pt] (11,13.5)-- (10,7);
    \draw [line width=0.5pt] (11,13.5)-- (10,9);
    \draw [line width=0.5pt] (11,13.5)-- (9,8);
    \draw [line width=0.5pt] (11,13.5)-- (11,7.5);
    \draw [line width=0.5pt] (11,13.5)-- (8,11);
    \draw [line width=0.5pt] (11,13.5)-- (14,5);
    \draw [line width=0.5pt] (11,13.5)-- (5.5,8);
    \draw [line width=0.5pt] (11,13.5)-- (16,8);

    \draw [line width=0.5pt] (8,11)-- (8,5);
    \draw [line width=0.5pt] (8,11)-- (11,3);
    \draw [line width=0.5pt] (5.5,8)-- (8,5);
    \draw [line width=0.5pt] (5.5,8)-- (11,3);
    \draw [line width=0.5pt] (8,11)-- (10,7);
    \draw [line width=0.5pt] (8,11)-- (10,9);
    \draw [line width=0.5pt] (8,11)-- (9,8);
    \draw [line width=0.5pt] (8,11)-- (11,7.5);
    \draw [line width=0.5pt] (5.5,8)-- (10,7);
    \draw [line width=0.5pt] (5.5,8)-- (10,9);
    \draw [line width=0.5pt] (5.5,8)-- (9,8);
    \draw [line width=0.5pt] (5.5,8)-- (11,7.5);

    \draw [line width=0.5pt] (8,5)-- (10,7);
    \draw [line width=0.5pt] (8,5)-- (10,9);
    \draw [line width=0.5pt] (8,5)-- (9,8);
    \draw [line width=0.5pt] (8,5)-- (11,7.5);

    \draw [line width=0.5pt] (8,5)-- (16,8);
    \draw [line width=0.5pt] (8,5)-- (14,5);

    \draw [line width=0.5pt] (11,3)-- (10,7);
    \draw [line width=0.5pt] (11,3)-- (10,9);
    \draw [line width=0.5pt] (11,3)-- (9,8);
    \draw [line width=0.5pt] (11,3)-- (11,7.5);
    \draw [line width=0.5pt] (11,3)-- (16,8);
    \draw [line width=0.5pt] (11,3)-- (14,5);

    \draw [line width=0.5pt] (16,8)-- (10,7);
    \draw [line width=0.5pt] (16,8)-- (10,9);
    \draw [line width=0.5pt] (16,8)-- (9,8);
    \draw [line width=0.5pt] (16,8)-- (11,7.5);
    \draw [line width=0.5pt] (14,5)-- (10,7);
    \draw [line width=0.5pt] (14,5)-- (10,9);
    \draw [line width=0.5pt] (14,5)-- (9,8);
    \draw [line width=0.5pt] (14,5)-- (11,7.5);

    \draw (19,7.5) node[anchor=north west] {$x^{2}$};
	\draw (19,8.8) node[anchor=north west] {$1$};
    \draw (21.2,7.5) node[anchor=north west] {$x^{6}$};
	\draw (21.2,8.8) node[anchor=north west] {$x^4$};
	
	\draw (6.8,5) node[anchor=north west] {$sr^{2}$};
    \draw (11,2.9) node[anchor=north west] {$sr^{6}$};
	\draw (14,12) node[anchor=north west] {$s$};
    \draw (11,14.5) node[anchor=north west] {$sr^4$};
	\draw (6.5,12) node[anchor=north west] {$sr$};
    \draw (4.5,10) node[anchor=north west] {$sr^5$};

	\draw (10,6) node[anchor=north west] {$r$};
    \draw (7,8) node[anchor=north west] {$r^5$};

	\draw (10,11) node[anchor=north west] {$r^{3}$};
	\draw (12,8) node[anchor=north west] {$r^{7}$};

	\draw (14,5) node[anchor=north west] {$sr^{3}$};
    \draw (15.9,8) node[anchor=north west] {$sr^{7}$};
	
	\begin{scriptsize}
	\draw [fill=black](20,8.5)circle (2.5pt);
    \draw [fill=black](20,7.5)circle (2.5pt); 
    \draw [fill=black](21,8.5)circle (2.5pt);
    \draw [fill=black](21,7.5)circle (2.5pt); 
	\draw [fill=black] (14.5,11) circle (2.5pt);
	\draw [fill=black] (8,11) circle (2.5pt);
    \draw [fill=black] (8,5) circle (2.5pt);
	\draw [fill=black](14,5)circle (2.5pt); 

    \draw [fill=black] (11,13.5) circle (2.5pt);
	\draw [fill=black] (5.5,8) circle (2.5pt);
    \draw [fill=black] (11,3) circle (2.5pt);
	\draw [fill=black](16,8)circle (2.5pt); 

    \draw [fill=black] (10,7) circle (2.5pt); 
	\draw [fill=black] (10,9) circle (2.5pt); 
    \draw [fill=black] (9,8) circle (2.5pt); 
	\draw [fill=black] (11,7.5) circle (2.5pt); 
	\end{scriptsize}
	\end{tikzpicture}
	\end{tabular}
\caption{Generating graph of $Q_3$, Generating graph of $Q_{4}.$}
	\label{fig:ga4}	
\end{figure}
\end{ex}
\subsection{The Quotient graph of $\Gamma(Q_n)$}
Let us first review the definition of  dihedral groups and introduce some notations. The Dihedral group \(D_n\) is defined as:
\[ D_n = \{r, s : r^n = e, s^2 = 1, srs = r^{-1}\} \]
where \(n \in \mathbb{N}\), and has order \(2n\). We consider the following three subsets of \(D_n\), which partition the group:
\[ \Omega_1 = \{r^i : \gcd(i, n) = 1\}, \]
\[ \Omega_2 = \{sr^i : 0 \leq i \leq n-1\}, \]
\[ \Omega_3 = \langle r \rangle \setminus \Omega_1. \] 
These sets are disjoint partition of $D,$ thus $$D_n=\Om_1\cup \Om_2\cup \Om_3.$$ These notations will be frequently used in this section. The presentation of a dihedral group includes \( s^2 = r^n=1 \), leading to the distinct structures. However there is a superficial resemblance between dicyclic and dihedral groups; both involve a kind of \textquotedblleft mirroring \textquotedblright of an underlying cyclic group.  Notably, \( Q_n \) is not a semidirect product of \( \langle x \rangle \) and \( \langle y \rangle \) because \( \langle x \rangle \cap \langle y \rangle = \{ y^2 \} \), which is not trivial. However, $D_n$ is the semidirect product of $\langle r\rangle$ ans $\langle s\rangle.$
The dicyclic group \( Q_n \) has a unique involution, an element of order 2 denoted by \( y^2 = x^n \), which resides in the center of \( Q_n \). The center \( Z(Q_n) \) comprises only the identity element and \( y^2 \). Modifying the presentation of \( Q_n \) by adding the relation \( y^2 = 1 \) yields the dihedral group \( D_n \). Consequently, the quotient group \( Q_n / Z(Q_n) \) is isomorphic to \( D_n \), implying that their generating graphs are isomorphic. The key question is whether the properties of the generating graph \(\Gamma(Q_n)\) can be deduced from the generating graph \(\Gamma(Q_n/Z(Q_n))\).

To address this, we extend the concept of the quotient structure of these groups to their generating graphs. The quotient graph of a graph is a fundamental idea for exploring this relationship.
First we quickly review the graph structure of $\Ga(D_n).$ Then we explore the concept of quotient graphs. For any $n>2,$ the $\text{diam}(\Ga(D_n))$ is 2. The set of vertex degrees is $\{0,2\varphi(n),n\}.$ The subgraphs induced by the sets $\Om_1,\Om_2$ and $\Om_3,$ denoted by $\Ga_{\Om_1},$ $\Ga_{\Om_2}$ and $\Ga_{\Om_3}$
 are $r_i-$regular graphs where $r_1=0,$ $r_2=\varphi(n)$ and $r_3=0$ respectively.
\begin{definition}
    Let $X$ be a graph, and $\pi=\{U_1,U_2,\dots,U_k\}$ be a partition of $V(X)$. The {\em quotient graph} $X/\pi$ is the graph where each cell $U_i$ we denote as a single vertex, and the adjacency is defined as; 
for any cells $U_i$, $U_j$ of $\pi$ where $i\neq j$, we add an edge between $U_i$ and $U_j$ in $X/\pi$ if there exists an adjacent pair of vertices $a \in U_i$ and
$b\in U_j$ in $X$. 
\end{definition}

We consider a relation $\equiv$ on $V(Q_n)$ as  follows.
$$x^a\equiv x^{n+a}\text{ and } x^{a}y\equiv x^{n+a}y \text{ where } 0\leq a\leq n-1.$$ Then, with respect to $\equiv,$ we have the following partition $\Theta$ on the vertex set $V(Q_n)=\{x^a,x^by: 0\leq a,b\leq 2n-1\},$
\begin{equation}\label{eq:theta}
    \Theta=\left\{\{1,x^{n}\},\{x,x^{n+1}\},\dots,\{x^{n-1},x^{2n-1}\},\{y,y^3\},\{xy,x^{1+n}y\},\dots,\{x^{n-1}y,x^{2n-1}y\}\right\}.
\end{equation} 
Note that each cell in the partition \( \Theta \) consists of a pair. With respect to this partition, we get the quotient graph of $\Ga(Q_n),$ denoted by $\Ga(Q_n)/\Theta.$ Also note that no two vertices inside a cell is adjacent, thus $\Ga(Q_n)/\Theta$ is a simple graph. Then, we have the following result.
\begin{theorem}
    Let $\Gamma(D_n)$ be the generating graph of the dihedral group $D_n$ and let $\Gamma(Q_n)$ be the generating graph of $Q_n.$ Then the quotient graph $\Gamma(Q_n)/\Theta$ of $\Gamma(Q_n)$ is isomorphic to the graph $\Gamma(D_n)$ for any $n>1.$ 
\end{theorem}
\begin{proof}
    Regarding the partition \( \Theta \) given in Equation~\ref{eq:theta}, we consider each equivalent pair of vertices in \( V(Q_n) \) as a single vertex in the quotient graph \( \Gamma(Q_n)/\Theta \) and  taking all the representatives of the equivalence classes of $\equiv$ as the vertices. The vertex set of  \( \Gamma(Q_n)/\Theta \) is given by
    $$V_\Theta=\left\{\overline {1},\overline {x},\dots, \overline {x^{n-1}},\overline {y},\overline {xy},\dots,\overline {x^{n-1}y}\right\}.$$ Clearly, $|V_\Theta|=2n.$ 
    
    Let $D_n=\langle r,s\rangle.$ Then $V(\D_n)=\{1,r,\dots,r^{n-1},sr,sr^2,\dots,sr^{n-1}\}.$ Now we define a map $f:V_\Theta\rightarrow V(D_n) \text{  by }$
    $$f(\overline {x^iy})= sr^i $$ where $0\leq i\leq n-1.$ Note that in $\Gamma(Q_n)/\Theta,$ the elements of the type $\overline {x^i}$ and $\overline {x^jy}$ are adjacent if and only if $\gcd(i,n)=1.$ WLOG if $k>l,$ then the pair of the type $\overline {x^ky}$ and $\overline {x^ly}$ are adjacent if  and only if $\gcd(k-l,n)=1.$ Then clearly $f$ preserves the adjacency of the vertices in the corresponding graphs. Hence $f$ is an isomorphism. This proves the required result.
\end{proof}
Moreover, we will demonstrate that the partition \( \Theta \) is equitable. To do this, let us first review the concept of an equitable partition. Consider a graph \( X \) with vertex set \( V(X) \) and edge set \( E(X) \).
\begin{definition}
A partition $\pi$ of $V(X)$ with cells $U_1, U_2, \ldots, U_k$ is equitable if the number
of neighbors in $U_j$ of a vertex $v \in U_i$ depends only on the choice of $U_i$ and $U_j.$ In
this case, the number of neighbors in $U_j$ of any vertex in $U_i$ is denoted $b_{ij}$ .
\end{definition}
We denote $\deg(g,U_j)=$ the number of vertices in the cell $U_j$ adjacent to the vertex $g\in U_i.$ Thus, \( \deg(g_1, U_j) = \deg(g_2, U_j) = b_{ij} \) for each $g_1,g_2\in U_i.$ 
 With respect to the partition $\pi$ of $V(X),$ the quotient matrix of the  graph is given by
  $$A(X/\pi)=[b_{ij}]$$ where each $b_{ij}$ corresponds to the cells $U_i$ and $U_j.$
  Now let us recall a result which is a consequence of equitable partitioning. To explore more about equitable partition one can refer to~\cite{godsil}.
  \begin{theorem}[Godsil and Royle,~\cite{godsil}]\label{factor}
    If $\pi$ is an equitable partition of a graph $X,$ then the characteristic polynomial of the quotient matrix $A(X/\pi)$ divides the characteristic polynomial of the adjacency matrix $A(X).$
 
  \end{theorem}
 With this understanding, we can proceed to establish the following result regarding the equitable nature of the partition \( \pi\).

The concept of an equitable partition of a vertex set in a graph can be generalized to any square matrix. Let \( M \) be an \( n \times n \) square matrix. Suppose the rows and columns of \( M \) can be partitioned into \( k \) subsets \( \{U_1, U_2, \ldots, U_k\} \) such that for each \( i \), the sum of the entries in each block \( M_{U_i, U_j} \) (the submatrix formed by the intersection of rows in \( U_i \) with columns in \( U_j \)) is constant for all \( j \). In other words, each block \( M_{U_i, U_j} \) has the same row sum for all rows in \( U_i \). This partition is known as an equitable partition.

The matrix \( \hat{Q} \), called the quotient matrix, is constructed where the \( (i,j) \)-th entry is the constant row sum of the block \( M_{U_i, U_j} \). A fundamental result in spectral graph theory is that the characteristic polynomial of the quotient matrix \( \hat{Q} \) divides the characteristic polynomial of the original matrix \( M \). Formally, if \( \phi_M(\lambda) \) and \( \phi_{\hat{Q}}(\lambda) \) denote the characteristic polynomials of \( M \) and \( \hat{Q} \), respectively, then:
\[ \phi_M(\lambda) = \phi_{\hat{Q}}(\lambda) \cdot p(\lambda) \]
for some polynomial \( p(\lambda) \). To know more about this concept, one can check in~\cite{bequi}.

This relationship implies that the eigenvalues of the quotient matrix \( \hat {Q} \) are also eigenvalues of the original matrix \( M \), with their multiplicities contributing to the overall structure of \( M \)'s spectrum. This elegant connection highlights the importance of equitable partitions in simplifying the analysis of the spectral properties of matrices and the graphs they represent.

\begin{theorem}
    Let $$ \Theta = \left\{\{1, x^n\}, \{x, x^{n+1}\}, \dots, \{x^{n-1}, x^{2n-1}\}, \{y, y^3\}, \{xy, x^{n+1}y\}, \dots, \{x^{n-1}y, x^{2n-1}y\} \right\} $$ be a partition of \( V(Q_n) \) in the generating graph \( \Gamma(Q_n) \). Then, \( \Theta \) is an equitable partition of \( \Gamma(Q_n) \).
\end{theorem}
 \begin{proof}
    Given the partition $$\Theta = \left\{\{1, x^n\}, \{x, x^{n+1}\}, \dots, \{x^{n-1}, x^{2n-1}\}, \{y, y^3\}, \{xy, x^{n+1}y\}, \dots, \{x^{n-1}y, x^{2n-1}y\} \right\},$$ we observe that no two elements within the same cell are adjacent. Moreover, if \( g_1 \) and \( g_2 \) belong to the same cell \( U_i \) and \( g_1 \) is adjacent to an element in \( U_j \), then \( g_2 \) will also be adjacent to that element. Therefore, for each pair of cells \( (U_i, U_j) \) in \(\Theta\) and for any \( g_1, g_2 \in U_i \), the following holds.
    \begin{description}
     \item[] If there is adjacency between the cells \( U_i \) and \( U_j \), then for $i\neq j,$
      \[
      \deg(g_1, U_j) = \deg(g_2, U_j) = 2.
      \]
     Otherwise,
     \item[]
      \[
      \deg(g_1, U_j) = \deg(g_2, U_j) = 0.
      \]
    \end{description}
    Thus, \(\Theta\) is an equitable partition of the generating graph \(\Gamma(Q_n)\).
 \end{proof}
 Given the distinct structures of dihedral and dicyclic groups, there can never be an isomorphism between them as groups. However, we will show that, in certain cases, their generating graphs can be isomorphic. This highlights an intriguing aspect where structural differences in algebraic groups do not preclude isomorphisms in their associated generating graphs.
\begin{theorem}
   Given $n\in \N,$ $\Gamma(Q_n)\cong\Gamma(D_{2n})$ if and only if $n$ is even.
   \end{theorem}
   \begin{proof}
   Considering \( D_{2n} \), we note that it has the same cardinality as \( Q_n \). However, they are not isomorphic as groups. Specifically, if we define a mapping that takes \( r \mapsto x \) and \( s \mapsto y \), the intersection \( \langle r \rangle \cap \langle s \rangle = \{1\} \) in \( D_{2n} \). In contrast, in \( Q_n \), the intersection \( \langle x \rangle \cap \langle y \rangle = \{1, y^2\} \), demonstrating a key structural difference between the two groups. With respect to their generating graphs, we first note that $\gcd(i,n)=1$ if and only if $\gcd(i,2n)=1$ when $n$ is even. When $n$ is odd, the number of isolated vertices in $\Gamma(Q_n)$ is $2n-2\varphi(n),$
       however, in $\Gamma(D_n)$ it is equal to $2n-\varphi(2n)=2n-\varphi(n).$ Thus both are unequal in this case. Thus, the graphs cannot be isomorphic when $n$ is odd.
    \end{proof}

\subsection{Graph properties of $\Delta(Q_n)$}
In this section, we will discuss graph-theoretic properties; for their standard definitions, refer to~\cite{west2001introduction}. To study the properties of the graph \(\Gamma(Q_n)\), it is sufficient to examine \(\Delta(Q_n)\).

Recall that the number of vertices in \(\Delta(Q_n)\) is \(2(n + \varphi(n))\) and the number of edges is \(6n\varphi(n)\). The degree of each vertex is either \(2n\) or \(2\varphi(n)\). The graph \(\Delta(Q_n)\) is not complete for any \(n \geq 2\) since \(|R_1| \geq 2\). The girth of \(\Delta(Q_n)\) is 3 for \(n \geq 2\) because \(x, y, xy\) form a triangle. Hence, \(\Delta(Q_n)\) is not bipartite.

The following properties are checked based on the known properties of the quotient graph of \(\Gamma(Q_n)\).
\begin{pro}
    \begin{enumerate}
        \item The graph $\Delta(Q_n)$ is Eulerian and Hamiltonian for $n> 1.$ 
        \item The graph $\Delta(Q_n)$ contains cycles starting at $s$  of all lengths from $3$ to $2(n+\varphi(n)).$  
        \item  The graph $\Delta(Q_n)$ is planar if and only if $n=2.$ 
           \item The domination number, $\gamma(\Delta(Q_n))$ and the total domination number, $\gamma_t(\Delta(Q_n))$  is $2$ for $n>1$. 
       \item The clique number $\omega(\Delta(Q_n)),$  is \(p+1\), where \(p\) is the smallest prime divisor of \(n\).
           \item The chromatic number, $\chi(\Delta(Q_n)),$ is equal to $\omega(\Delta(Q_n)).$ 
            \item The independence number, \(\alpha(\Delta(D_n))\), of the graph \(\Delta(D_n)\) is \(\frac{n}{p}\), where \(p\) is the smallest prime factor of \(n\).

       \end{enumerate}
\end{pro}
\begin{proof}
    \begin{enumerate}
        \item Since degrees are all even for any $n>1,$ thus the graph is Eulerian. To show it is Hamiltonian, first note that for each $0\leq i \leq (2n-1),$ there is an edge between and $x^{i}y$ and $x^{i+1}y.$ 
        Let $U(n)= \{a_1,a_2,\dots,a_{\varphi(n)}\}$ and $n+U(n)=\{n+a_1,n+a_2,\dots,n+a_{\varphi(n)}\}$ such that $a_1<a_2<\dots<a_{\varphi(n)}<n+a_{1}<\dots<n
        +a_{\varphi(n)}.$ Then $y-x^{a_1}-x^{a_1}y-x^{a_1+n}-x^{n+a_1+1}y-\dots-x^{n+a_2-1}y-x^{a_2}-x^{a_1+1}y-x^{a_1+2}y-\dots-x^{a_2-1}y-x^{n+a_2}-x^{n+a_2}y-x^{a_3}y-\dots-x^{n+a_{\varphi(n)}}-x^{n-1}y-x^{n}y-x^{2n-1}y-y$ is a  Hamiltonian path in $\Delta(Q_n).$
        \item This has been given in the paper~\cite{planar}. Here, we are giving a proof for the sake of completeness. It is easy to see that $Q_2$ is planar. If $n>2$, then $2\varphi(n)\ge 4$  and one can see that $K_{3,3}$ as a subgraph in $\Delta(Q_n).$
        \item The set $\{x,y\}$ can be seen as the least dominating set. Clearly, we cannot further reduce it. 
        \item Note that \(\Gamma(Q_n) / \Theta \cong \Delta(D_n)\). It is given that the  \(\omega(\Delta(D_n))\) is \(p+1\) for any \(n > 1\). Thus, $\om(\Delta(Q_n))\geq (p+1).$
        The only possibility to obtain a larger clique is to add an element to a set of the form \(\mathcal{W} = \{x^{i_1} y, x^{i_2} y, \dots, x^{i_{p+1}} y\}\), which forms a clique.  Note that any two elements in \(\mathcal{W}\) belong to distinct equivalence classes and are of the form \(i_j = t_i p + l_i\) for some \(t_i \in \mathbb{N}\) and \(0 \leq l_i < p\). Adding \(x^m y\) for some \(0 \leq m \leq 2n - 1\) to the set \(\mathcal{W}\) will result in adjacency to one of the elements in \(\mathcal{W}\). Therefore, \(\mathcal{W}\) is of maximal size, and the clique number of \(\Delta(Q_n)\) is also \(p+1\).
        \item In the case of $Q_{n},$ the fitting height is at most 2. Thus $\omega=\chi$ in $\Delta(Q_n),$ the result follows from the Proposition~2.5 of \cite{lucchini2009clique}.
        \item In the quotient graph, each vertex represents a pair of vertices from the original graph that are not adjacent. Thus, the size of the independent set in \(\Delta(Q_n)\) is twice that in \(\Delta(D_n)\). Therefore, \(\alpha(\Delta(Q_n)) = \frac{2n}{p}\), where \(p\) is the smallest prime divisor of \(n\).

    \end{enumerate}
\end{proof}

From the above discussion, we have gained an understanding of the structural properties of the graph \(\Gamma(Q_n)\). One can explore more graph-theoretical properties of \(\Gamma(Q_n)\).

\section{Spectrum of adjacency and Laplacian matrix of $\Ga(Q_n)$}\label{sec4:spec}
In this section, we will determine the spectrum of the associated matrices of \(\Gamma(Q_n)\), in particular adjacency and Laplacian matrix. We will utilize the quotient isomorphism with the generating graph of dihedral groups as a tool to find the complete spectrum. Since the adjacency and Laplacian spectra have already been computed for \(D_n\) in~\cite{Me} the spectrum of \(\Gamma(Q_n)\) can be derived from it. 

We aim to first find the spectrum of the adjacency matrix, denoted by \(A(Q_n)\), which we will discuss in Subsection~\ref{sec4:AM}. Once the spectrum of \(A(Q_n)\) is determined, computing the spectrum of the Laplacian matrix, denoted by \(L(Q_n)\), becomes straightforward. Thus, we will cover the spectrum of the Laplacian matrix in Subsection~\ref{sec4:SLM}.

\subsection{Spectrum of the Adjacency Matrix}\label{sec4:AM}
In this section, our objective is to determine the spectrum of the adjacency matrix of \(\Gamma(Q_n)\). To achieve this, we will employ certain binary graph operations to simplify the structure of \(\Gamma(Q_n)\). We will demonstrate that finding the spectrum of the adjacency matrix of \(\Gamma(Q_n)\) can be reduced to finding the spectrum of the adjacency matrix of the induced subgraph \(\Gamma_{\Omega}\). 

Before delving into this, we will introduce some preliminary concepts of graph operations that will be used throughout this section.

The {\em union} of two graphs \(X_1\) and \(X_2\), denoted by \(X_1 \cup X_2\), is the graph whose vertex set is \(V(X_1) \cup V(X_2)\) and whose edge set is \(E(X_1) \cup E(X_2)\). The {\em join} of \(X_1\) and \(X_2\), denoted by \(X_1 \vee X_2\), is the graph obtained from \(X_1 \cup X_2\) by adding all possible edges between the vertices of \(X_1\) and \(X_2\).
To better understand these operations, consider the following example.
\vglue 1mm
\begin{table}[!ht]
\centering
 \begin{tabular}{|c|c|c|c|}
 \hline
\quad\quad $X_1$ &\quad \quad$X_2$ & \quad\quad$X_1\cup X_2$ &\quad \quad $X_1\vee X_2$\\
 \hline
\begin{tikzpicture}
    \vertex (1) at (1,2) {1};
    \vertex (2) at (2,2) {2};
    \vertex (3) at (1.5,3){3};
    \path[-]
    (1) edge (2)
    (1) edge (3);
\end{tikzpicture} & 
 \begin{tikzpicture}
    \vertex (1) at (2,2) {4};
    \vertex (2) at (3,3) {5};
    \vertex (3) at (4,2) {6};
    \path[-]
    (1) edge (2)
    (2) edge (3)
    (3) edge (1);
\end{tikzpicture} & 
\begin{tikzpicture}
    \vertex (1) at (2,2) {1};
    \vertex (2) at (4,2) {2};
    \vertex (3) at (3,3) {3};
    \path[-]
    (2) edge (1)
    (3) edge (1);
\end{tikzpicture} 
 \begin{tikzpicture}
    \vertex (1) at (2,3) {4};
    \vertex (2) at (3,4) {5};
    \vertex (3) at (4,3) {6};
    \path[-]
    (1) edge (2)
    (3) edge (1)
    (3) edge (2);
\end{tikzpicture} &
\begin{tikzpicture}
    \vertex (1) at (2,2) {1};
    \vertex (2) at (4,2) {2};
    \vertex (3) at (3,1.5) {3};
    \vertex (4) at (3,4) {5};
    \vertex (5) at (4,3) {6};
    \vertex (6)  at (2,3) {4};
    \path[-]
    (1) edge (2)
    (1) edge (4)
    (2) edge (4)
    (5) edge (6)
    (1) edge (5)
    (1) edge (6)
    (2) edge (5)
    (2) edge (6)
    (4) edge (6)
    (4) edge (5)
    (3) edge (4)
    (3) edge (5)
    (3) edge (6)
    (1) edge (3);
    
\end{tikzpicture}\\
\hline
\end{tabular}
\vglue 1mm
\caption{{Union and Join of graphs $X_1$ and $X_2.$}}\label{tab:oper1}
\end{table}

Based on the above discussion on graph operations, it can be observed that the generating graph \(\Gamma(Q_n)\) can be easily represented as follows:
\begin{equation}\label{Eq:def}
\Ga(Q_n)=(\Ga_{\Om}\vee \Ga_{R_1})\cup \Ga_{R_2}.
\end{equation}
Note that $\Ga_{\Om}$ is a $2\varphi(n)$-regular graph and $\Ga_{R_1},$ $\Ga_{R_2}$ are the empty graphs. Let us relabel the vertices of $\Ga(Q_n)$ in the rows and columns of $A(Q_n)$ as per the expression $\Ga(Q_n)=(\Ga_{\Om}\vee \Ga_{\R_1})\cup \Ga_{R_2},$ so we have
$$A(Q_n)=\begin{tabular}{c|ccc}
         & $\Om$ & $R_1$ & $R_2$\\
         \hline
         $\Om$ & $A_{11}$ & $A_{12}$ & $A_{13}$\\
         $R_1$ & $A_{21}$ & $A_{22}$ & $A_{23}$\\
         $R_2$ & $A_{31}$ & $A_{32}$ & $A_{33}$\\
       \end{tabular},$$ where $A_{11},A_{22}, A_{33}$ are the adjacency matrices of $\Ga_{\Om}, \Ga_{R_1}$ and $ \Ga_{R_2}$ respectively.
Let $O_{m\times n}$ and $J_{m\times n}$ be the matrices of size  $m\times n$ with all entries $0$ and $1$ respectively. Then from  Equation~\ref{eq:nbd}, we have
$$A(Q_n)=\begin{bmatrix}
        A_{11} & J_{2n\times 2\varphi(n)} & O_{2n\times (2(n-\varphi(n)))}\\
        J_{2\varphi(n)\times 2n } & O_{2\varphi(n)\times 2\varphi(n)} & O_{2\varphi(n)\times 2(n-\varphi(n))}\\
        O_{2(n-\varphi(n))\times n} & O_{2(n-\varphi(n))\times2\varphi(n)} & O_{2(n-\varphi(n))\times 2(n-\varphi(n))}
       \end{bmatrix}.$$
       All entries of the blocks of \( A(Q_n) \) are explicitly known, except for the block \( A_{11} \). Our goal is to explicitly understand the adjacency relations of the vertices within \( \Omega \), represented by the matrix \( A_{11} \). To achieve this, we will use the quotient graph.

In the last section, we see that the quotient graph $\Ga(Q_n)/\Theta$ is isomorphic to $\Ga(D_n).$  For notational convinence, we will use same notations for $\Ga(Q_n)/\Theta$ as we defined of the $\Ga(D_n).$ We consider an {equivalence} relation on $\Om_{2}$ defined by $\overline{x^iy}\sim \overline{x^jy}$ if and only if $N(\overline{x^iy})=N(\overline{x^jy}),$ where $N(\overline{x^iy}),$ $N(\overline{x^jy})$ are the neighborhood sets. Then,
 for each $\overline{x^iy},$ the corresponding equivalence class is given by 
 \begin{equation}\label{eq:sim1}
    \left[\overline {x^iy}\right]=\left\{\overline {x^iy},\overline {x^{i+n_0}y},\overline {x^{i+2n_0}y},\ldots, \overline {x^{i+(\frac{n}{n_0}-1)n_0}y}\right\}.
 \end{equation}
 where $n_0$ is the square free part of $n$ (see~\cite{Me}).
 Note that $\left|\left[\overline {x^iy}\right]\right|=\frac{n}{n_0}.$ From Theorem~[4.2,~\cite{Me}] it is an equitable partition of the set $\Om_{2}.$ In the below lemma, we can see that we can extend the equivalence relation to the set $\Omega\subset V(Q_n).$

\begin{lemma}\label{sec4:lem1}
    The relation $\sim$ is an equivalence relation on $\Om \subset V(Q_n).$ For each  $x^iy,$ the equivalence class is given by  $$\left[x^iy\right]=\left\{x^iy,x^{i+n}y,x^{i+n_0}y,x^{i+n_0+n}y, x^{i+2n_0}y,x^{i+2n_0+n}y\ldots, x^{i+(\frac{n}{n_0}-1)n_0+n}y\right\}.$$ No two elements in a class are adjacent. Moreover, each class has size $\frac{2n}{n_0},$ and have exactly $n_0$ disjoint classes.  
\end{lemma}
\begin{proof}
   The partition $\Theta,$ gives the partition of $\Om,$ denoted by $\Om_{\Theta}$ as follows: $$\Om_{\Theta}=\left\{\{y,y^3=x^ny\},\{xy,x^{1+n}y\},\dots,\{x^{n-1}y,x^{2n-1}y\}\right\}.$$ It is clear that for each set  inside $\Om_{\Theta},$ which is a pair of vertices, is glued together as a single vertex in the quotient graph and denoted the set by $\Om_2$. Moreover, it is easy to see that $$B_i=B_{n+i}$$ for each pair $\{x^iy, x^{n+i}y\}$ of $\Om_{\Theta}$ where $0\leq i\leq (n-1)$ (see Equation~\ref{eq:Ba}). Thus, $N(x^iy)=N(x^{n+i}y).$ Thus, using Equation~\ref{eq:sim1}, corresponding to each $x^iy,$ the equivalence class with respect to $\sim$ in $\Om$ is given by 
   $$\left[x^iy\right]=\left\{x^iy,x^{i+n}y,x^{i+n_0}y,x^{i+n_0+n}y, x^{i+2n_0}y,x^{i+2n_0+n}y\ldots, x^{i+(\frac{n}{n_0}-1)n_0+n}y\right\}.$$
Since no two elements in any pair inside $V_\Theta$ are adjacent, thus for the set $\left[x^iy\right].$ Clearly, \begin{equation}\label{eq:2times}
    \left|\left[x^iy\right]\right|=2\left|\left[\overline{x^iy}\right]\right|=\frac{2n}{n_0}.
\end{equation}  Each class have the same size and since $|\Om|=2n,$ thus there are exactly $n_0$ distinct equivalence classes. It completes the proof.
\end{proof}

\begin{lemma}\label{sec4:lem2}
     For any two equivalence class $\left[x^{i}y\right]$ and $\left[x^jy\right]$ of $\Om$ either their union induces a $K_{\left(\frac{2n}{n_0},\frac{2n}{n_0}\right)},$ a complete bipartite graph or its complement graph.
\end{lemma}
\begin{proof}
   From the previous lemma, no two vertices of class  $\left[x^{i}y\right]$ and $\left[x^{j}y\right]$ are adjacent.
   Next, for any $u\in \left[x^{i}y\right]$ if there exists $v\in\left[x^{j}y\right]$ such that they are adjacent in $\Ga_{\Om}$, then $u$ will be adjacent to all the vertices which are inside $\left[x^{j}y\right].$ Thus, as a result we obtain that $\left\{\left[x^{i}y\right],\left[x^{j}y\right]\right\}$ is a bipartition of the graph obtain by $\left[x^{i}y\right]\cup \left[x^{j}y\right]$ which either induces a complete bipartite graph or the complement of it, of size $\frac{2n}{n_0}\times \frac{2n}{n_0}.$
\end{proof}
\begin{theorem}
    Let $\Ga_{\Om}$ be the subgraph of $\Ga(Q_n),$ induced by $\Om.$ Then the relation $\sim$ on $\Om$ gives an equitable partition. 
\end{theorem}
 \begin{proof} 
Let $\pi=\{[x],[xy],[x^2y],\ldots,[x^{n_0-1}y]\}$ be a partition of $\Om.$ By Lemma~\ref{sec4:lem1} and \ref{sec4:lem2},  for each $g_1,g_2\in [x^ay]$
 $$\deg(g_1,[x^by])=\deg(g_2,[x^by])=\frac{2n}{n_0}$$ 
 when we have adjacency among the cells. Otherwise,
$$\deg(g_1,[x^by])=\deg(g_2,[x^by])=0.$$ Thus, $\pi$ is an equitable partition of the graph $\Ga_{\Om}$.

 \end{proof}
\begin{theorem}\label{sec4:thm-q1}
Let $A_{11}/\Theta$ be the adjacency matrix of the subgraph, induced by $\Om_{\Theta}$ in $\Ga(Q_n)/\Theta.$ Let $\widetilde{A_{11}}$ and $\widetilde{A_{11}/\Theta}$ be the quotient matrices of the graph $\Ga_{\Om}$ and $\Ga_{\Om_2}$ with respect to the relation $\sim$ respectively. Then 
   \begin{equation}
    \widetilde{A_{11}}=2\widetilde{A_{11}/\Theta}
   \end{equation}
\end{theorem}
\begin{proof}
    From the Lemma~\ref{sec4:lem2}, it is clear that if $[x^iy]\cup[x^jy]$ induces the graph $K_{2m,2m}$ where $m=\frac{n}{n_0}$ if and only if $\overline{[x^iy]}\cup\overline{[x^jy]}$ induces the graph  $K_{m,m}.$ Equivalently, induces the the complement graphs. The quotient matrices of the graph $\Ga_{\Om}$ and $\Ga_{\Om_{2}}$ with respect to the relation $\sim,$ are of the same size $n_0\times n_0.$ However, the scalars $b_{ij}$  correspond to the cells $U_i$ and $U_j$ differs in both the matrices (see~\ref{eq:2times}).
    \begin{equation}
        \widetilde{A_{11}}=2\widetilde{A_{11}/\Theta}
       \end{equation}
       Hence it proves the required statement.
\end{proof}

\begin{theorem}[\cite{spectra}]\label{sec4:thm-s}
    For $i=1,2$, let $X_i$ be a $k_i$ regular graph on $n_i$ vertices with adjacency eigenvalues $\lambda_{i,1}=k_i\geq\lambda_{i,2}\geq \dots\geq \lambda_{i,n_i}.$ The adjacency spectrum of $X_1\vee X_2$ consists of $\lambda_{i,j_i}$ for $i=1,2$ and $2\leq j_i\leq n_i,$ and two more eigenvalues of the form 
   $$\frac{(k_1+k_2)\pm\sqrt{(k_1-k_2)^2+4n_1n_2}}{2}.$$
 \end{theorem}
\begin{theorem}[\cite{Me}]\label{sec4:thm-q2}
        Let $1=d_1<d_2<\ldots<d_{2^k}=n_0$ be the divisors of $n_0.$ Let $A(\Om_2)$ be the adjacency matrix corresponds to $\Ga(\Om_2).$ Then, the spectrum of $A(\Om_2)$ is given by
        $$Spec_{A(\Om_2)}=
          \begin{pmatrix}
         0&\varphi(n)&\mu(d_2)\frac{\varphi(n)}{\varphi(d_2)}&
         \mu(d_3)\frac{\varphi(n)}{\varphi(d_3)}&\dots
         &\mu(d_{2^k})\frac{\varphi(n)}{\varphi(d_{2^k})}\\
         n-n_0&1&\varphi(d_2)&\varphi(d_3)&\dots&\varphi(d_{2^k})
        \end{pmatrix}
      $$ 
      where $\mu(\cdot)$ denotes M\"{o}bius function.
\end{theorem}
\begin{rem}
    In the previous theorem, we take $d_i's$ as the divisors of $n_0,$ so they all are square free integers. Thus we can write $$\mu(d_i)\frac{\varphi(n)}{\varphi(d_i)}=(-1)^{k_{d_i}}\frac{\varphi(n)}{\varphi(d_i)},$$
since $\mu(d_i)=(-1)^{k_{d_i}}$ where $k_{d_i}$ is the number of distinct prime factors of $d_i.$
\end{rem}

 \begin{theorem}[\cite{Me}]\label{sec4:thm-a}
Let $\Ga(D_n)$ be the generating graph of $D_n$ and let $k$ be the number of distinct prime divisors of $n.$ Suppose $1=d_1<d_2<\ldots<d_{2^k}=n_0$ are the divisors of $n_0.$ Then the spectrum of $A(D_n)$ is given by $${Spec}_{A(D_n)}=
\begin{pmatrix}
   0&\mu(d_2)\frac{\varphi(n)}{\varphi(d_2)}&
   \mu(d_3)\frac{\varphi(n)}{\varphi(d_3)}&\dots
   &\mu(d_{2^k})\frac{\varphi(n)}{\varphi(d_{2^k})}&\lambda_1&\lambda_2\\
   2n-(n_0+1)&\varphi(d_2)&\varphi(d_3)&\dots&\varphi(d_{2^k})&1&1
\end{pmatrix},$$
where $\mu(\cdot)$ denotes M\"{o}bius function and $$\lambda_1=\frac {1}{2}\left(\varphi(n)+\sqrt{\varphi(n)^2+4n\varphi(n)}\right),\,
\lambda_2=\frac{1}{2}\left(\varphi(n)- \sqrt{\varphi(n)^2+4n\varphi(n)}\right).$$ 
\end{theorem}
Now we compute the spectrum for the matrix $A_{11},$ the adjacency matrix of the subgraph $\Ga_{\Om}.$
\begin{theorem}\label{thm:s-om}
    Let $1=d_1<d_2<\ldots<d_{2^k}=n_0$ be the divisors of $n_0.$ Let $A_{11}$ be the adjacency matrix corresponds to $\Ga_x{\Om}.$ Then, the spectrum of $A_{11}$ is given by 
    \begin{equation}\label{thm:a11}
        Spec_{A_{11}}=
       \begin{pmatrix}
      0&2\varphi(n)&2\mu(d_2)\frac{\varphi(n)}{\varphi(d_2)}&
      2\mu(d_3)\frac{\varphi(n)}{\varphi(d_3)}&\dots
      &2\mu(d_{2^k})\frac{\varphi(n)}{\varphi(d_{2^k})}\\
      2n-{n_0}&1&\varphi(d_2)&\varphi(d_3)&\dots&\varphi(d_{2^k})
   \end{pmatrix}
\end{equation}
   where $\mu(\cdot)$ denotes M\"{o}bius function.
   \end{theorem}
   \begin{proof}
   There are exactly $n_0$ non-zero eigen values, which we get by applying Theorem~\ref{sec4:thm-q1} and Theorem~\ref{sec4:thm-q2}. Since the size of $A_{11}$ is $2n\times 2n$, thus there are exactly $2n-n_0$ zero eigenvalues. Hence we get the required values.
   \end{proof}

 Now we can conclude the spectrum of the adjacency matrix for the dicyclic groups.
\begin{theorem}
    Let $\Ga(Q_n)$ be the generating graph of $Q_n.$ 
    Let $1=d_1<d_2<\ldots<d_{2^k}=n_0$ be the divisors of $n_0.$ Then the spectrum of $A(Q_n)$ is given by $$Spec_{A(Q_n)}=
       \begin{pmatrix}
      0&2\mu(d_2)\frac{\varphi(n)}{\varphi(d_2)}&
      2\mu(d_3)\frac{\varphi(n)}{\varphi(d_3)}&\dots
      &2\mu(d_{2^k})\frac{\varphi(n)}{\varphi(d_{2^k})}&\lambda_1&\lambda_2\\
      4n-(n_0+1)&\varphi(d_2)&\varphi(d_3)&\dots&\varphi(d_{2^k})&1&1
   \end{pmatrix},$$
   where $\mu(\cdot)$ denotes M\"{o}bius function and $$\lambda_1=\varphi(n)+\sqrt{\varphi(n)^2+4n\varphi(n)},\quad
   \lambda_2=\varphi(n)- \sqrt{\varphi(n)^2+4n\varphi(n)}.$$ 
   \end{theorem}
   \begin{proof}
       Using above theorem, Theorem ~\ref{sec4:thm-s} and using the fact that eigenvalues of the union of two graphs is the union of their eigenvalues. It proves the required statement.
   \end{proof}

\subsection{Spectrum of the Laplacian Matrix}\label{sec4:SLM}
In this section, we will see the Laplacian matrix of the graph $\Ga(Q_n).$ Let $D=\text{diag}(\alpha_{1},\alpha_{2},\ldots,\alpha_{2n})$ where $\alpha_i's$ represent the vertex degrees of $\Ga(Q_n),$ called the diagonal matrix. Let $L(Q_n)$ denotes the Laplacian matrix of $\Ga(Q_n)$ and it is defined as $L(Q_n)=D-A(Q_n).$ 

\begin{theorem}[\cite{spectra}]\label{thm:Lsp}
    For $i=1,2$, let $X_i$ be a $k_i$ regular graph on $n_i$ vertices with laplacian eigenvalues $\lambda^L_{i,1}\geq\lambda^L_{i,2}\geq \dots\geq \lambda^L_{i,n_i}=0.$ Then the Laplacian spectrum of $X_1\vee X_2$ consists of $\lambda^L_{1,j_i}+n_{2}$ and $\lambda^L_{2,j_i}+n_{1}$ for $2\leq j_i\leq n_i,$ and two more eigenvalues are $0$ and $n_1+n_2.$
 \end{theorem}

Let us first state the theorem for the Laplacian spectrum of the graph $\Ga(D_n).$
\begin{theorem}[\cite{Me}]
    Let $\Ga(D_n)$ be the generating graph of $D_n.$ Then the spectrum of $L(D_n),$ denoted by $Spec_{L(D_n)}$ is given by
    $$\begin{pmatrix}
   0& 2\varphi(n)-\mu(d_2)\frac{\varphi(n)}{\varphi(d_2)}&\cdots&2\varphi(n)-\mu(d_{\tau(n_0)})\frac{\varphi(n)}{\varphi(d_{\tau(n_0)})}&2\varphi(n)&n&n+\varphi(n)\\
   n-\varphi(n)+1&\varphi(d_2)&\cdots&\varphi(d_{\tau(n_0)})&n-n_0&\varphi(n)-1&1
\end{pmatrix}.$$
\end{theorem}
Using the fact that the quotient graph of the Dicyclic group concerning the partition $\Theta$ is isomorphic to the dihedral group we can conclude the following result.
\begin{theorem}
    Let $\Ga(Q_n)$ be the generating graph of $Q_n.$ Let $\lambda$ be an eigenvalue of $L(D_n),$ 
    then $2\lambda$ is an eigenvalue of $L(Q_n).$ Explicitly, the spectrum of $L(Q_n),$ denoted by $Spec_{L(Q_n)},$ is given as
    $$\begin{pmatrix}
   0& 4\varphi(n)-2\mu(d_2)\frac{\varphi(n)}{\varphi(d_2)}&\cdots&4\varphi(n)-2\mu(d_{\tau(n_0)})\frac{\varphi(n)}{\varphi(d_{\tau(n_0)})}&4\varphi(n)&2n&2(n+\varphi(n))\\
   2(n-\varphi(n))+1&\varphi(d_2)&\cdots&\varphi(d_{\tau(n_0)})&2n-n_0&2\varphi(n)-1&1
\end{pmatrix}.$$
\end{theorem}
\begin{proof}
    Let $\psi_{1},\ldots ,\psi_{2n}$; $\lambda_{1},\ldots, \lambda_{2\varphi(n)}$; $\gamma_{1},\ldots ,\gamma_{2(n-\varphi(n))}$ are the eigenvalues of $A_{11},A_{22},A_{33}$ respectively. Now we find the spectrum of the respective Laplacian matrices. By the definition of Laplacian matrix, the eigenvalues of $L(\Ga_{\Om}), L(\Ga_{R_1})$ and $L(\Ga_{R_2})$  are  $2\varphi(n)-\psi_{j_1},$ where $1\leq j_1\leq 2n$; $-\lambda_{j_2},$ where $1\leq j_2\leq 2\varphi(n)$; $-\gamma_{j_3},$ where $1\leq j_3\leq 2(n-\varphi(n))$ respectively.  Using Theorem~\ref{thm:Lsp}, we can get the whole spectrum of $L(Q_n).$ \par 
To follow the above discussion, we first determine the spectrum of $L
(\Ga_{\Om})$ and $L(\Ga_{R_1}).$ We denote $\sigma_1$  and $\sigma_2$ as their respective spectra. Since $\Ga(R_1)$ is an empty graph, thus the adjacency eigenvales are all zero. Thus, we only need to figure out the spectrum of $L
(\Ga_{\Om}).$ Note that $L((\Ga_{\Om}))=2\varphi(n)I-A_{11},$ and we know that all the eigenvalues of $A_{11}.$ Notice that all its eigenvaues are two times the eigenvalues of $\varphi(n)I-A_{11}/\Theta=L(\Om_2)$ (see Theorem~\ref{sec4:thm-q1}). Let us summarise it in the following table.
\begin{center}
 \begin{tabular}{|c|c|}
  \hline
$L(\Ga_{R_1})={0-A_{22}}$
 &\quad\quad $L(\Ga_{\Om})=2L(\Ga_{\Om_2})$ \\
  \hline
  \hline
  &\\
 $\quad\quad\sigma_1=\begin{pmatrix}
    0\\
    2\varphi(n)
\end{pmatrix}$
& $\sigma_2=
\begin{pmatrix}
    2\varphi(n)-2\mu(d_1)\frac{\varphi(n)}{\varphi(d_1)}&\cdots&2\varphi(n)-2\mu(d_{\tau(n_0)})\frac{\varphi(n)}{\varphi(d_{\tau(n_0)})}&2\varphi(n)\\
   1&\cdots&\varphi(d_{\tau(n_0)})& 2n-n_0 
\end{pmatrix}$\\
&\\
\hline
\end{tabular}
  \end{center}
  \vglue 1mm
Note that $2\mu(d_1)\frac{\varphi(n)}{\varphi(d_1)}=2\varphi(n).$ Thus, $2\varphi(n)-2\mu(d_1)\frac{\varphi(n_0)}{\varphi(d_1)}=0.$ It is a well-known result that the multiplicity of the eigenvalue 0 of the Laplacian matrix equals the number of connected components in the graph. Since $\Ga_{\Om}$ is connected, thus the multiplicity of 0 is 1. To get the eigenvalues of join of the two graphs, we apply Theorem~\ref{thm:Lsp}. The spectrum  of ${L(\Ga_{\Om}{\vee}\Ga_{R_1})},$ denoted by $Spec_{\Delta(Q_n)},$ is given by 
$$ \begin{pmatrix}
   0& 4\varphi(n)-2\mu(d_2)\frac{\varphi(n)}{\varphi(d_2)}&\cdots&4\varphi(n)-2\mu(d_{\tau(n_0)})\frac{\varphi(n)}{\varphi(d_{\tau(n_0)})}&4\varphi(n)&2n&2(n+\varphi(n))\\
   1&\varphi(d_2)&\cdots&\varphi(d_{\tau(n_0)})&2n-n_0&2\varphi(n)-1&1
\end{pmatrix}.$$
It is known that the eigenvalues of the Laplacian of the union of two graphs are the union of the eigenvalues of both graphs. Using this fact and $\Ga_{R_2}$ is a $0$- regular graph of size  $2(n-\varphi(n))$, we get the required result.
\end{proof}
       
\section{Distance and Eccentricity Spectrum}\label{sec:EDspec}
\subsection{Distance spectrum}
The distance matrix is a matrix representation of graphs in algebraic graph theory, defined similarly to the adjacency matrix. Suppose \(X\) is a connected graph with the set of vertices \(V(X) = \{v_1, v_2, \dots, v_n\}\), and let \(d_{ij} = d(v_i, v_j)\), where \(d(v_i, v_j)\) represents the shortest path length between the vertices \(v_i\) and \(v_j\). To avoid any ambiguity, we denote the distance matrix of \(X\), by \(Dis(X)\), which is an \(n \times n\) matrix whose rows and columns correspond to the vertices, with the \((i, j)\)th entry being \(d_{ij}\).

In this section, we will determine the spectrum of the distance matrix for dicyclic and dihedral groups. Given that the adjacency and Laplacian spectra are already known for dihedral groups (cite), we are now interested in deriving the distance spectrum using the adjacency spectrum, which is generally challenging to find.

Let us define the distance matrix for dihedral groups and introduce some notations. We consider the connected component of \(\Gamma(D_n)\). By removing the elements of \(\Omega_3\) from the vertex set \(V(D_n)\), we obtain the induced subgraph \(\Delta(D_n)\). The distance matrix of \(\Delta(D_n)\), in a block form is given by.
$$Dis(D_n)=\begin{tabular}{c|cc}
    & $\Om_1$ & $\Om_2$ \\
    \hline
    $\Om_1$ & $2(J-I)$ & $2J$\\
    $\Om_2$ & $2J$ & $2(J-I)-A(\Omega_2)$\\
  \end{tabular}, $$ where $A{(\Omega_2)}$ is the adjacency matrix of $\Ga_{\Om_2},$ which is the graph induced by $\Om_2,$ $J$ is the matrix with all entries $1$ and $I$ is the identity matrix.
  To determine the spectrum of the distance matrix of $\Delta(D_n),$ we will use the following result. In this result, it is deduced that it is possible to compute the distance spectrum from the adjacency spectra for the join of two regular graphs.
  \begin{theorem}\cite{dspectra}\label{thm:ds}
    For $i=1,2$, let $X_i$ be a $k_i$ regular graph on $n_i$ vertices with adjacency eigenvalues $\lambda_{i,1}=k_i\geq\lambda_{i,2}\geq \dots\geq \lambda_{i,n_i}.$ The distance spectrum of $X_1\vee X_2$ consists of $-\lambda_{i,j_i}-2$ for $i=1,2$ and $2\leq j_i\leq n_i,$ and two more eigenvalues of the form 
    $$n_1+n_2-2-\frac{k_1+k_2}{2}\pm\sqrt{\left(n_1-n_2-\frac{k_1-k_2}{2}\right)^2+n_1n_2}.$$
  \end{theorem}
  
Since the connected component of $\Ga(D_n)$ is the join of the subgraphs induced by the set $\Om_2$ and the set $\Om_1$. The adjacency matrix $Az(\Om_1)$ is a zero matrix of order $\varphi(n)\times \varphi(n)$, thus all its adjacency eigenvalues are zero with multiplicity $\varphi(n).$ Thus, we conclude the following result for the distance spectrum of $\Delta(D_n)$.
\begin{theorem}
    Let $\Delta(D_n)$ be the generating graph of $D_n$ after removing the isolated vertices. Then the distance spectrum, denoted by $Spec_{Dis(D_n)}$ is given by 
    $$
    \begin{pmatrix}
   -2&-2-\mu(d_2)\frac{\varphi(n)}{\varphi(d_2)}&
   -2-\mu(d_3)\frac{\varphi(n)}{\varphi(d_3)}&\dots
   &-2-\mu(d_{2^k})\frac{\varphi(n)}{\varphi(d_{2^k})}&\lambda_1&\lambda_2\\
   \varphi(n)+n-n_0-1&\varphi(d_2)&\varphi(d_3)&\dots&\varphi(d_{2^k})&1&1
\end{pmatrix}$$
where $$\lambda_i=\left(n-2+\frac{\varphi(n)}{2}\right)\pm\sqrt{n^2+\left(\frac{3\varphi(n)}{2}\right)^2-2n\varphi(n)}.$$
\end{theorem}
\begin{proof}
    It is important to note that the connected component $\Delta(D_n)$ of $\Ga(D_n)$ is induced by the set $\Om_1\cup\Om_2$ which is a join of the graphs $\Ga_{\Om_1}$ and $\Ga_{\Om_2}.$ The spectra of the repective adjacency matrices are known to us (see Theorem~\ref{sec4:thm-q2}). Thus using above theorem we get the required result.
\end{proof}
Similarly, we have the following result for the dicyclic groups.
\begin{theorem}
    Let $\Delta(Q_n)$ be the generating graph of $Q_n$ after removing the isolated vertices. Then the distance spectrum, denoted by $Spec_{Dis(Q_n)},$ is given by 
    $$
    \begin{pmatrix}
   -2&-2-2\mu(d_2)\frac{\varphi(n)}{\varphi(d_2)}&
   -2-2\mu(d_3)\frac{\varphi(n)}{\varphi(d_3)}&\dots
   &-2-2\mu(d_{2^k})\frac{\varphi(n)}{\varphi(d_{2^k})}&\lambda_1&\lambda_2\\
   2(\varphi(n)+n)-n_0-1&\varphi(d_2)&\varphi(d_3)&\dots&\varphi(d_{2^k})&1&1
\end{pmatrix}$$
where $$\lambda_i=\left(\varphi(n)+2(n-1)\right)\pm\sqrt{\left(3\varphi(n)-2n\right)^2+4n\varphi(n)}.$$
\end{theorem}
\begin{proof}
    It is important to note that the connected component $\Delta(Q_n)$ of $\Ga(Q_n)$ is induced by the set $\Om\cup R_1$ which is a join of the graphs $\Ga_{\Om}$ and $\Ga_{R_1}.$ The spectra of the respective adjacency matrices are known to us (see Theorem~\ref{thm:s-om}). Thus applying Theorem~\ref{thm:ds} we get the required result.
\end{proof}

\subsection{Eccentricity Spectrum}
A distance-type matrix for connected and simple graphs, referred to as the eccentricity matrix, was introduced by M. Randić in \cite{WANG}. This matrix can be viewed as the opposite of the adjacency matrix, which can be constructed from the distance matrix by selecting, for each row and each column, only the smallest distances corresponding to adjacent vertices.

For a graph \( X \), the {\em eccentricity} \( e(u) \) of a vertex \( u \) is given by \( e(u) = \max\{d(u,v) : v \in V(X)\} \). The diameter of a graph is defined as
\[ \text{diam} = \max\{e(u) : u \in V(X)\}. \]

Now, we define the {\em eccentricity} matrix \(\mathcal{E}(X)\) of a graph \( X \) with rows and columns indexed by the vertices. The entries are defined as follows:
\begin{equation*}
 \E_{X}(u,v)=\begin{dcases}
    d(u,v) &\mbox{if $d(u,v)=\, \min\{e(u), e(v)\}$,}\\
    0 & \mbox{if $d(u,v)<\, \min\{e(u), e(v)\}$.}\\
    \end{dcases}
\end{equation*}
The eccentricity matrix is symmetric, and thus the eigenvalues of \( \mathcal{E}(X) \) are real. In the context of generating graphs for groups, the eccentricity of a vertex holds significant implications for group generation. The center of a graph, consists of vertices with the minimum eccentricity, meaning they have the smallest maximum distance to other vertices. In the case of \( \Delta(D_n) \), since the diameter is 2, if there exists a vertex $v$ having $e(u)=1,$ then the vertices within the center of the graph can generate the entire group with any other vertices, forming a clique inside $\Ga(D_n).$ Let us determine the eccentricity matrix of the connected component of \( \Gamma(D_n) \).

We consider the connected component $\Delta(D_n)$ of $\Gamma(D_n)$, whose vertex set is $\Omega_1\cup\Omega_2.$ Clearly, the diameter of the $\Delta(D_n)$ is 2.   
Thus, for any vertex $u\in\Omega_1\cup\Omega_2,$ $e(u)\leq 2.$
In particular, for $n>2,$ $e(u)=2$ for all $u\in \Omega_1$ However, for $v\in \Om_2,$ $e(v)=1$ or $2$ when $n=p$ or $n\neq p$ respectively, where $p$ is a prime. 

The eccentricity matrix of $D_n,$ denoted by $\E(D_n),$ and $\E_n(u,v)$ denotes the entry corresponds to vertices $\{u,v\}.$ It is as follows.
\begin{enumerate}
  
\item When $n$ is a prime
\begin{equation*}
    \E_n(u,v)=\begin{dcases}
        2&\mbox{if $u, v \in \Omega_1,$ }\\
        1&\mbox{if $u, v \in \Omega_2,$ or $u\in\Omega_1,$ $v\in\Omega_2,$ }\\
        0&\mbox{ if $u=v$}\\
      \end{dcases}
\end{equation*}
\item When $n$ is not a prime
\begin{equation*}
    \E_n(u,v)=\begin{dcases}
        2&\mbox{if $u, v \in \Omega_1$ or $u, v\in \Omega_2$ such that $u\notedge v,$}\\
        0&\mbox{otherwise.}\\
      \end{dcases}
\end{equation*}
\end{enumerate}
The matrix can be seen in the block form as follows.
\begin{enumerate}
   
  \item When $n$ is a prime
  \begin{equation}\label{eq:ecc1}
    \E(D_n)=\begin{tabular}{c|cc}
    & $\Om_1$ & $\Om_2$ \\
    \hline
    $\Om_1$ & $2(J-I)$ & $J$\\
    $\Om_2$ & $J$ & $(J-I)$\\
  \end{tabular}
\end{equation}
 where $J$ is the matrix with all entries $1$ and $I$ is the identity matrix.
  \item When $n$ is not a prime
  \begin{equation}\label{eq:ecc2}
    \E(D_n)=\begin{tabular}{c|cc}
      & $\Om_1$ & $\Om_2$ \\
      \hline
      $\Om_1$ & $2(J-I)$ & $0$\\
      $\Om_2$ & $0$ & $2(J-I-A_{22})$\\
    \end{tabular}
\end{equation} where $A(\Om_2)$ is the adjacency matrix of $\Ga_{\Om_2}$ and $J$ is the matrix with all entries $1$ and $I$ is the identity matrix.
\end{enumerate}

\begin{theorem}
    The characteristic polynomial of the matrix $\E(D_n)$ is given by 
    \begin{enumerate}
        \item When $n$ is a prime $$C_{D_n}(x)=\left(x^2-(3p-5)x+(p-1)(p-4)\right)(x+2)^{(p-2)}(x+1)^{(p-1)}.$$
        \item When $n$ is not a prime $$C_{D_n}(x)=(x+2)^{(\varphi(n)+n-n_0-1)}(x-\varphi(n))\prod\limits_{d/n_0}\left(x-\left((-1)^{k_d+1}\frac{2\varphi(n)}{\varphi(d)}-2\right)\right)^{\varphi(d)}.$$
    \end{enumerate}

\end{theorem}
\begin{proof}
\begin{enumerate}
    \item Let $n=p,$ where p is a prime. Since $V(\Delta(D_n))=\Omega_1\cup \Omega_2,$ which is a disjoint partition of the vertex set with sixe $\varphi(n)$ and $n$ repectively. Note that the row sum of each of the blocks are equal. Thus it is an equitable partition of the vertex set. Hence the associated quotient matrix is of size $2\times 2,$ and it is given by
\begin{equation*}
    M_n=
\begin{pmatrix}
    2(p-2)&p\\
    p-1&p-1
\end{pmatrix}.
\end{equation*}
Clearly, the  characteristic polyomial of $M_n$ is given by 
\begin{equation}
    \phi(x)=x^2-(3p-5)x+(p-1)(p-4).
\end{equation}
 Note that $\phi(x)$ is a factor of the characteristic polynomial of $\E(D_n).$ To figure out all the other eigenvalues, first note that the vectors $\begin{pmatrix} v_1\\0\end{pmatrix}$ and $\begin{pmatrix} 0\\v_2\end{pmatrix}$ where $v_1$ is an eigenvector corresponding to $\lambda=-2$ for the block $2(J-I)$ of sie $(p-1)\times (p-1)$ with mutiplicity $(p-2)$ and $v_2$ corresponding to $\mu=-1$ for the block $(J-I)$ of size $p\times p$ with mutiplicity $(p-1)$ are eigenvectors of $M_n$ corresponding to $\lambda$ and $\mu$ respectively.  Thus the rest of the factors of the polynomial are $(x+2)^{(p-2)}(x+1)^{(p-1)}.$ Since the order of $\E(D_n)$ is $2p-1\times 2p-1,$ thus the characteristic polynomial of $\E(D_n)$ is as follows.
\begin{align*}
    C_{D_n}(x)&=\phi(x)(x+2)^{(p-2)}(x+1)^{(p-1)}\\
    &=\left(x^2-(3p-5)x+(p-1)(p-4)\right)(x+2)^{(p-2)}(x+1)^{(p-1)}.
\end{align*}
\item
Let us assume that $n$ is not a prime. Note that from Equation~\ref{eq:ecc2}, $\E(D_n)$ is a block diagonal matrix. So we use the fact that the characteristic polynomial of a block diangonal matrix is the product of the characteristic polynomial of the respective blocks.Thus, we have
\begin{equation}
    C_{D_n}(x)=C_{\Omega_1}(x)C_{\Omega_2}(x)
\end{equation} 
where $C_{\Omega_1}(x)$, $C_{\Omega_2}(x)$ are the characteristic polynomial of the blocks $2(J-I)$  of order $\varphi(n)\times\varphi(n)$ and $2(J-I-A(\Om_2))$ of size $n\times n.$ It is clear that 
\begin{equation}
    C_{\Omega_1}(x)=(x+2)^{(\varphi(n)-1)}(x-\varphi(n)).\end{equation} 
The main task is to determine  $C_{\Omega_2}(x).$ First note that \begin{equation}\label{eq:char1}
    C_{\Om_2}(x)=\det(\left(2(J-I-A(\Om_2))\right)-x I)=\det(\left(2(J-A(\Om_2))\right)-(x+2)I).
\end{equation} Let $y=x+2,$ then the polynomial can be written in terms of y as below.
$$C_{\Om_2}(y)=\det\left(2(J-A(\Om_2))-yI\right).$$
Observe that finding the eigenvalues of the matrix $Q=J-A(\Om_2)$  will be sufficient to resolve the Equation~\ref{eq:char1}.

Notice that, in the matrix $Q,$ the entries are either 0 or 1 depends on whether entries of $A(\Om)$ is 1 and 0 respectively. Clearly, it retains the matrix row sum consistency. Thus, we can partition $Q$ in a similar way as $A(\Om).$ Since the quotient matrix of $A(\Om)$ is circulant, thus the quotient matrix of $Q$, denoted by $\widetilde{Q}$ is a circulant matrix with the first row $[b_0\,b_1\,\dots\,b_{n_0-1}]$ where for $0\leq i\leq ({n_0}-1),$ $b_i$ is given by  
$$b_i=\begin{dcases}
 		0& \quad \text{if }\gcd(i,n_0)=1;\\
 	    \frac{n}{n_0}&\quad\text{otherwise.}
 		\end{dcases}$$
         the associated polynomial $q_{n}(t)$ is given by $$q_{n}(t)=\sum\limits_{i=0}^{n_0-1} b_it^i=\left(\frac{n}{n_0}\right)\left(\sum\limits_{\substack{0\leq i\leq ({n_0}-1)\\ \gcd(i,n_0)=1}}t^i + \sum\limits_{j\neq i}t^j\right)=\frac{-n}{n_0}\sum\limits_{\substack{0\leq j\leq ({n_0}-1)\\ \gcd(j,n_0)= 1}}t^j.$$
         ~\\
         We can express  ${\widetilde{Q}}=q_{n}(P),$ where $P$ is a permutation matrix 
 with $P^{n_0}=I$ and $I$ is the identity matrix. Thus we can conclude that the eigenvalues of ${\widetilde{Q}}$ are $q_{n}(\zeta_{n_0}^l)$ where $l\in \{0,1,\ldots,n_0-1\}$ and $\zeta_{n_0}$ is a primitive $n_0$-{th} root of unity. 
Note that we can express the sum of the $l-$th power of primitive $n_0$-th roots of unity as follows.
 $$c_{n_0}(l)=\mu(d)\frac{\varphi(n_0)}{\varphi(d)},$$
 where $d=\frac{n_0}{\gcd(n_0,l)}$ and $\mu(\cdot)$ is the M\"{o}bius function.  
 Thus, $$-\frac{n}{n_0}c_{n_0}(l)=-\mu(d)\frac{\varphi(n)}{\varphi(d)}=(-1)^{k_d+1}\frac{\varphi(n)}{\varphi(d)},$$
 here we are using the fact that $n_0$ is a square free integer, therefore $\mu(d)=(-1)^{k_d}$ where $k_d$ is the number of distinct prime factors of $d.$
Thus, the characteristic polynomial of $Q$ is given by $$C_Q(y)=C_{\Omega_2}(y)= y^{n-n_0}\prod\limits_{d/n_0}\left(y-(-1)^{k_d+1}\frac{2\varphi(n)}{\varphi(d)}\right)^{\varphi(d)},$$ where $k_d$ is the number of distinct prime factors of $d.$ 
Since $y=x+2,$ thus we concluded the following
$$C_{\Om_2}(x)=(x+2)^{n-n_0}\prod\limits_{d/n_0}\left(x-\left((-1)^{k_d+1}\frac{2\varphi(n)}{\varphi(d)}-2\right)\right)^{\varphi(d)}.$$ 
\end{enumerate}
\end{proof}
Let us determine the eccentricity matrix of $\Ga(Q_n).$ WLOG we consider the connected component $\Delta(Q_n)$ of $\Gamma(D_n)$, whose vertex set is $R_1\cup\Omega.$ Clearly, the diameter of the $\Delta(Q_n)$ is 2. Thus, for any vertex $u\in R_1\cup\Omega,$ $e(u)\leq 2.$ In particular, for $n>2,$ $e(u)=2$ for all $u\in R_1$ or $u\in \Om_2.$ Since for every $u=x^iy\in \Om,$ where $0\leq i\leq (2n-1),$ there is $v=x^{n+i\text{(mod n)}}y\in \Om$ such that $\langle u,v\rangle\neq G.$ The eccentricity matrix of $Q_n,$ denoted by $\E(Q_n),$ is given by
\begin{equation*}
    \E_n(u,v)=\begin{dcases}
        2&\mbox{if $u,v \in R_1, \text{ or } u,v \in \Omega \text{ such that } u\nsim v;$ }\\
        0&\mbox{otherwise.}\\
      \end{dcases}
\end{equation*}
The matrix can be seen in the block form as follows.
  $$\E(D_n)=\begin{tabular}{c|cc}
      & $R_1$ & $\Om$ \\
      \hline
      $R_1$ & $2(J-I)$ & $0$\\
      $\Om$ & $0$ & $2(J-I-A_{22})$\\
    \end{tabular}, $$ where ${A_{22}}$ is the adjacency matrix of $\Ga_{\Om_2}$ and $J$ is the matrix with all entries $1$ and $I$ is the identity matrix.

\begin{theorem}
The characteristic polynomial of the eccentricity matrix is given by
    $$C_{Q_n(x)}=C_{Q_n(x)}=\left(x-(4\varphi(n)+2)\right)(x+2)^{2(\varphi(n)+n)-(n_0+1)}\prod\limits_{d/n_0}\left(x-\left((-1)^{k_d+1}\frac{4\varphi(n)}{\varphi(d)}-2\right)\right)^{\varphi(d)}.$$
\end{theorem}
\begin{proof}
The matrix $\E(Q_n)$ is block diagonal matrix. Thus the characteristic polynomial is given by 
$$C_{Q_n}(x)=C_{R_1}(x)C_{\Omega}(x)$$ where $C_{R_1}(x)$, $C_{\Omega}(x)$ are the characteristic polynomial of the blocks $2(J-I)$  of size $2\varphi(n)\times2\varphi(n)$ and $2(J-I-A_{11})$ of size $2n\times 2n.$ It is clear that 
$$C_{R_1}(x)=(x+2)^{(2\varphi(n)-1)}(x-(4\varphi(n)-2)).$$ 
As in the above theorem, we can get the polynomial $C_{\Omega_2}(x).$ 
Since we have the relation on the quotient matrix of $A_{11}$ and $$A(\Om_2)$$(see Theorem~ \ref{sec4:thm-q1}). Thus we have 
\begin{equation*}
    \widetilde {A_{11}}=2\widetilde{A_{\Om_2}}
\end{equation*}
Thus, the characteristic polynomial of $Q$ is given by $$C_Q(y)=C_{\Omega}(y)= y^{2n-n_0}\prod\limits_{d/n_0}\left(y-(-1)^{k_d+1}\frac{4\varphi(n)}{\varphi(d)}\right)^{\varphi(d)},$$ where $k_d$ is the number of distinct prime factors of $d.$ 
Since $y=x+2,$ thus we concluded the following
$$C_{\Om}(x)=(x+2)^{2n-n_0}\prod\limits_{d/n_0}\left(x-\left((-1)^{k_d+1}\frac{4\varphi(n)}{\varphi(d)}-2\right)\right)^{\varphi(d)}.$$ 
Thus, we have 
$$C_{Q_n(x)}=\left(x-(4\varphi(n)+2)\right)(x+2)^{2(\varphi(n)+n)-(n_0+1)}\prod\limits_{d/n_0}\left(x-\left((-1)^{k_d+1}\frac{4\varphi(n)}{\varphi(d)}-2\right)\right)^{\varphi(d)}$$
    \end{proof}
In the above theorems, for $\Delta(D_n)$ and $\Delta(Q_n),$ we explicitely give all the eigen values. Thus, we can conclude the complete spectrum.

\section* {Statements and Declarations}
\subsection* {Competing interests} The authors declare that they have no competing interests.
\subsection* {Author contributions} All the authors have the same amount of contribution.
\subsection* {Funding} The second named author is supported by Shiv Nadar Institution of Eminence Ph.D. Fellowship, 
\subsection*{Availability of data and materials} Data sharing does not apply to this article as no data sets
were generated or analyzed during the current study.
\subsection*{Acknowledgement}
The authors express their sincere gratitude to the learned referee for her/his meticulous reading and valuable suggestions which have improved the quality of the article.

\addcontentsline{toc}{section}{Bibliography}
\bibliographystyle{siam}
\bibliography{Refer}

\end{document}